\documentclass[11pt]{amsart}
\usepackage{amsmath,amssymb,amsfonts,amsthm}
\usepackage[a4paper,margin=29mm]{geometry}

\usepackage{xcolor}
\usepackage{tikz}
\usepackage{float}
\usepackage[unicode,hidelinks]{hyperref}

\newtheorem{theorem}{Theorem}[section]
\newtheorem{lemma}[theorem]{Lemma}
\newtheorem{proposition}[theorem]{Proposition}
\theoremstyle{definition}

\theoremstyle{remark}

\numberwithin{equation}{section}
\DeclareMathOperator{\Ric}{Ric}
\DeclareMathOperator{\Scal}{Scal}
\DeclareMathOperator{\Vol}{Vol}
\DeclareMathOperator{\tr}{tr}
\DeclareMathOperator{\divg}{div}
\DeclareMathOperator{\Id}{Id}
\newcommand{\R}{\mathbb R}
\newcommand{\Sph}{\mathbb S}
\newcommand{\dd}{\,\mathrm d}

\hypersetup{pdftitle={Scalar curvature integrals under nonnegative Ricci curvature},pdfauthor={}}
\title[A four-dimensional counterexample]{Unbounded normalized scalar curvature integrals in dimension four}

\author{Haoxuan Cheng}
\address{School of Mathematical Sciences, Fudan University,
Shanghai 200433, China}
\email{hxcheng25@m.fudan.edu.cn}

\date{}
\hypersetup{pdftitle={Unbounded normalized scalar curvature integrals in dimension four}}
\hypersetup{pdfauthor={Haoxuan Cheng}}
\hypersetup{
  pdfsubject={Counterexamples to Yau's scalar curvature integral question},
  pdfkeywords={Nonnegative Ricci curvature, scalar curvature integrals, pole, volume collapse}
}
\begin{document}
\begin{abstract}
We give counterexamples to Yau's question on normalized scalar curvature integrals in every dimension $n\geq4$. For each such $n$, there exists a smooth complete Riemannian metric $g$ on $\R^n$ with nonnegative Ricci curvature and a pole such that
\[
 \lim_{R\to\infty}R^{2-n}\int_{B_q(R)}\Scal_g\,\mathrm dV_g=+\infty
\]
for every fixed $q\in\R^n$. Here $B_q(R)$ denotes the geodesic ball of radius $R$ centered at $q$, and $\Scal_g$ is the scalar curvature of $g$.
\end{abstract}
\maketitle

\section{Introduction}\label{sec:intro}

Let $(M^n,g)$, $n\geq3$, be a connected, smooth, complete noncompact Riemannian manifold without boundary satisfying $\Ric_g\geq0$, and fix $p\in M$. We write $B_p(R)=\{x:d_g(p,x)<R\}$ and use the natural Riemannian measure $\dd V_g$. Our sign convention makes the unit sphere have sectional curvature $+1$, and $\Scal_g=\tr_g\Ric_g$. The Ricci inequality is understood as an inequality of quadratic forms. Define
\begin{equation}\label{eq:normalizedmass}
 \mathcal Q_{g,p}(R)=R^{2-n}\int_{B_p(R)}\Scal_g\dd V_g.
\end{equation}
The normalization is invariant under constant rescaling: for $\lambda>0$,
\[
 \mathcal Q_{\lambda^2g,p}(\lambda R)=\mathcal Q_{g,p}(R).
\]
Distance, scalar curvature and volume scale by $\lambda$, $\lambda^{-2}$ and $\lambda^n$, respectively.
\par
The question is whether every fixed triple $(M,g,p)$ satisfying these assumptions also satisfies
\begin{equation}\label{eq:question}
 \sup_{R\geq1}\mathcal Q_{g,p}(R)<\infty.
\end{equation}
The bound may depend on the fixed triple $(M,g,p)$.

Yau~\cite[p.~278, Problem~9]{Yau1992} asked whether the following quantities remain bounded as $R\to\infty$:
\[
 R^{-n+2k}\int_{B_p(R)}\sigma_k(\Ric_g)\dd V_g.
\]
Here $\sigma_k(\Ric_g)$ is the $k$th elementary symmetric polynomial in the eigenvalues of $g^{-1}\Ric_g$; in particular, $\sigma_1(\Ric_g)=\Scal_g$. Completeness and Hopf--Rinow~\cite[Thm.~6.19 and Corollary 6.21]{Lee2018} imply that closed bounded balls are compact. Smoothness of the curvature then makes the integral finite for each radius. For $k=1$, boundedness as $R\to\infty$ is equivalent to \eqref{eq:question}.

Xu~\cite{Xu2024} formulates the scalar case as the existence of a finite ordinary limit, with no pole assumption. The example below also rules out that stronger conclusion. Yang~\cite{Yang2013} gives higher-order counterexamples among complete K\"ahler metrics on $\mathbb C^m$, of real dimension $2m$, for $2\leq k\leq m$; the present construction concerns the scalar case $k=1$.

A point $p$ is called a \emph{pole} if the exponential map $\exp_p:T_pM\to M$ is a global diffeomorphism; this is the definition used in Xu~\cite{Xu2024}. The counterexample below has this additional property.

Write $\omega_n$ for the volume of the Euclidean unit ball in $\R^n$. When the following limit exists, we call it the asymptotic volume ratio at $p$:
\begin{equation}\label{eq:avr}
 \operatorname{AVR}_p(g)=\lim_{R\to\infty}\frac{\Vol_g B_p(R)}{\omega_nR^n}.
\end{equation}
\relax{}

For complete three-manifolds with nonnegative Ricci curvature and a pole, Xu~\cite{Xu2024} proves that $\lim_{R\to\infty}\mathcal Q_{g,q}(R)=8\pi(1-\operatorname{AVR}_p(g))$ for every fixed $q\in M$. The theorem below shows that the corresponding boundedness conclusion fails in dimension four under the same Ricci and pole assumptions. Proposition~\ref{prop:product} gives this failure in every higher dimension.

\begin{theorem}\label{thm:main}
There exist a smooth complete Riemannian metric $G$ on $\R^4$ and a pole $p$ such that $\Ric_G\geq0$ everywhere and $\Ric_G>0$ outside a compact set. For every fixed $q\in\R^4$, one has $\operatorname{AVR}_q(G)=0$ and
\begin{equation}\label{eq:main}
 \lim_{R\to\infty}R^{-2}\int_{B_q(R)}\Scal_G\dd V_G=+\infty.
\end{equation}
\end{theorem}

\subsection*{The geometric mechanism}

We construct toric metrics on $\Sph^3$ with a uniform positive Ricci lower bound, volume tending to zero, and total scalar curvature $I(k)=\int_{\Sph^3}\Scal_k\dd V_k$ tending to infinity. Their orbit space is an interval, with one circle factor collapsing at each endpoint (Figure~\ref{fig:mechanism}). One surviving core circle has circumference of order $A$, while the volume is at most of order $A^{-3}$. The long core contributes a term of order $A$ to the scalar integral through an endpoint flux identity. The diameters nevertheless remain uniformly bounded by the Bonnet--Myers theorem~\cite[Thm.~12.24]{Lee2018}.

We place the slices on a radial end with metric $G=\mathrm dr^2+\rho(r)^2\bar k_{s(r)}$, where $\rho(r)>0$ is comparable to $r$, and $\bar k_s$ is a path of metrics on the fixed three-sphere. Write $S(s)=I(\bar k_s)$.

To preserve Ricci curvature along the end, we use the slow-variation method of Colding--Naber~\cite{ColdingNaber2013}. Their lemma assumes a constant slice volume form, a positive intrinsic Ricci lower bound, and uniform bounds on two parameter derivatives and one spatial covariant derivative of the first parameter derivative. Our slice volumes decrease to zero. We instead arrange a common volume form multiplied by a spatially constant contracting factor, whose logarithmic derivative enters the radial Ricci estimate with a favorable sign. After reparameterizing to bound the deformation derivatives, we let the path evolve at a double-logarithmic rate along a slightly concave end. The concavity and small cone opening provide enough radial and tangential Ricci curvature to absorb the deformation and mixed terms. We arrange that the path is initially round and complete the end with a rotationally symmetric cap.

Intrinsic scalar curvature and slice volume scale by $\rho^{-2}$ and $\rho^3$, leaving a factor of $\rho$ in the scalar integral. The curvature estimates below retain half of this intrinsic contribution. For all sufficiently large $R$, integration over a shell then yields
\begin{equation}\label{eq:geometric-scaling}
 R^{-2}\int_{B_p(R)}\Scal_G\dd V_G
 \geq \frac{c_0}{R^2}\int_{R/2}^R rS(s(r))\dd r
 \geq\frac{3c_0}{8}\inf_{r\in[R/2,R]}S(s(r)),
\end{equation}
where $c_0>0$ is fixed. The right-hand side tends to infinity because $s(r)\to\infty$ and $S(s)\to\infty$; monotonicity of $S$ is unnecessary. Proposition~\ref{prop:scalartail} supplies the slice estimate, and Section~\ref{sec:completion} proves that $r$ is distance from the center.

In a recent preprint, Hao--Zhu~\cite[Thm.~1.1 and Remark 1.2]{HaoZhu2026} construct a complete metric with strictly positive Ricci curvature on $\R^3$ and obtain examples with nonnegative Ricci curvature in higher dimensions by Euclidean products.

The asymptotic conclusions of Hao--Zhu~\cite[Thm.~1.1]{HaoZhu2026} and Theorem~\ref{thm:main} are, for every fixed basepoint,
\[
\begin{aligned}
\text{Hao--Zhu }(n=3):\qquad &
\limsup_{R\to\infty}\frac{1}{R}\int_{B_p(R)}\Scal_g\dd V_g=+\infty,\\[4pt]
\text{This paper }(n=4):\qquad &
\lim_{R\to\infty}\frac{1}{R^2}\int_{B_q(R)}\Scal_G\dd V_G=+\infty.
\end{aligned}
\]

Their three-dimensional example also has zero asymptotic volume ratio~\cite[Proposition 1.3]{HaoZhu2026}. Unlike our examples, it and its Euclidean products have no pole: Xu's three-dimensional theorem rules out a pole in the first factor, and the exponential map of a product is the product of the exponential maps. In dimension four, our metric has a positive definite Ricci tensor outside a compact set, whereas their direct Euclidean product has a zero Ricci direction everywhere.

Hao--Zhu describe their approach as an adaptation of the iterative gluing construction used by Wu--Yan for Milnor's conjecture~\cite[Introduction]{HaoZhu2026}. They smooth a toric cap and core, cut out a three-ball with convex boundary, and normalize its boundary meridian while preserving a region with large scalar curvature integral~\cite[Section 2]{HaoZhu2026}. These regions are inserted into successive annular blocks, with each new region chosen after the preceding geometric cost is known~\cite[Section 4]{HaoZhu2026}. In our closed three-sphere construction, the coefficient correction in Section~\ref{sec:slices} makes the second circle collapse smoothly while retaining a uniform positive Ricci lower bound. Both integral estimates use the tangential Ricci flux identity, applied to different domains and endpoint data~\cite[Lemmas 2.12 and 2.14]{HaoZhu2026}.

\subsection*{Related results}
The classical motivation is the Cohn--Vossen inequality, which bounds total Gaussian curvature on complete noncompact surfaces of nonnegative Gaussian curvature~\cite[Satz 6 and Satz 8, pp.~79--80]{CohnVossen1935}. In higher dimensions, affirmative results for $\mathcal Q_{g,p}(R)$ are known under additional geometric assumptions. Under nonnegative sectional curvature, Petrunin's unit-ball estimate~\cite[Thm.~1.1]{Petrunin2009}, applied to $R^{-2}g$, gives a uniform upper bound. In the locally conformally flat setting, Ma~\cite[Thm.~1.3]{Ma2026} proves that the normalized scalar integral has a finite ordinary limit under nonnegative Ricci curvature. Here local conformal flatness means that the metric is locally a positive smooth scalar multiple of a Euclidean metric.

In dimension three, the geodesic-ball integral has also been controlled under a pole assumption or two-sided scalar curvature bounds. For complete three-manifolds with $\Ric\geq0$ and a pole, Zhu~\cite{Zhu2022} proves the asymptotic upper bound $20\pi$; Xu's exact limit formula in the same setting was stated above. For complete three-manifolds with nonnegative Ricci curvature and scalar curvature bounded above and below by positive constants, Munteanu--Wang~\cite[Thm.~1.2]{MunteanuWang2025} obtain the sharp upper bound $8\pi$ for the limsup.

Green functions give estimates for related three-dimensional quantities with different domains or weights. Xu~\cite{Xu2020} obtains an estimate on sublevel sets of the reciprocal of a minimal positive Green function, with a gradient weight, under nonparabolicity and nonnegative Ricci curvature. For complete noncompact nonparabolic three-manifolds with nonnegative scalar curvature, one end and vanishing first Betti number, Munteanu--Wang~\cite{MunteanuWang2023} prove a level-set monotonicity formula when the minimal positive Green function tends to zero at infinity.

On complete noncompact K\"ahler manifolds, another line of work estimates the scalar integral divided by the ball volume, under holomorphic bisectional curvature assumptions. Chen--Zhu~\cite{ChenZhu2005} obtain linear average decay under positive holomorphic bisectional curvature, and Ni--Tam~\cite[Thm.~0.4(ii)]{NiTam2003} obtain it when that curvature is nonnegative and positive at one point. Quadratic average decay is known under pinching or volume-growth assumptions: Shi--Yau~\cite[Thm.~1]{ShiYau1996} prove it in complex dimension at least three when holomorphic bisectional curvature is bounded, nonnegative and uniformly pinched below by a positive multiple of the scalar curvature; Ni~\cite[Corollary 1]{Ni2005} proves it under bounded nonnegative holomorphic bisectional curvature and maximal volume growth, meaning positive asymptotic volume ratio. Both the volume normalization and these additional curvature hypotheses distinguish these results from the question studied here.

Local curvature estimates also depend on the geometric assumptions. Under the two-sided bound $|\Ric|\leq n-1$ and the lower volume bound $\Vol B_p(1)>v_0>0$, Jiang--Naber~\cite{JiangNaber2021} bound the average of $|\mathrm{Rm}|^2$ over $B_p(1)$ by a constant depending only on $n$ and $v_0$, where $\mathrm{Rm}$ is the full Riemann curvature tensor. The large-scale rescalings of our four-dimensional metric have unit-ball volumes tending to zero, so this noncollapse assumption does not hold uniformly for them.

Section~\ref{sec:slices} constructs the compact slices. Section~\ref{sec:path} joins them to the round metric, establishes a common volume gauge and controls parameter derivatives. Section~\ref{sec:ambient} proves the radial curvature estimates. Section~\ref{sec:completion} completes the metric, proves Theorem~\ref{thm:main} and gives the product examples.

\section{Toric three-spheres with large total scalar curvature}\label{sec:slices}

We use metrics invariant under the standard $\Sph^1\times\Sph^1$ action on $\Sph^3$, with both angular variables of period $2\pi$. Fix the differentiable model
\[
 \Sph^3=\{(\sin x\,e^{i\theta},\cos x\,e^{i\phi}):
 0\leq x\leq\pi/2\},\qquad
 \gamma=\mathrm dx^2+\sin^2x\,\mathrm d\theta^2+\cos^2x\,\mathrm d\phi^2.
\]
At $x=0$ the $\theta$ circle collapses and the $\phi$ circle remains; at $x=\pi/2$ the roles reverse. These are the two core circles of the action. The tensor $\gamma$ is the unit round metric and $\Vol(\gamma)=2\pi^2$. For any metric $h$ on $\Sph^3$, write $I(h)=\int_{\Sph^3}\Scal_h\dd V_h$ for its total scalar curvature. We will construct a family $h_A$ with $\Ric(h_A)\geq h_A$, $\Vol(h_A)\to0$, and $I(h_A)\to\infty$ as $A\to\infty$.

In arclength coordinates on the orbit interval, the metric takes the form $\mathrm dt^2+a(t)^2\mathrm d\theta^2+b(t)^2\mathrm d\phi^2$. The function $a$ starts at zero with unit slope and reaches $A$ at the other endpoint; $b$ starts at $B_A$ and ends at zero with slope $-1$. The coefficients below force $B_A$ to decay fast enough that the volume tends to zero. At the second endpoint, the values $a=A$ and $\mathrm db/\mathrm dt=-1$ give the growing term in the scalar integral. The proofs below establish the smoothness conditions and the flux identity that yields this term.

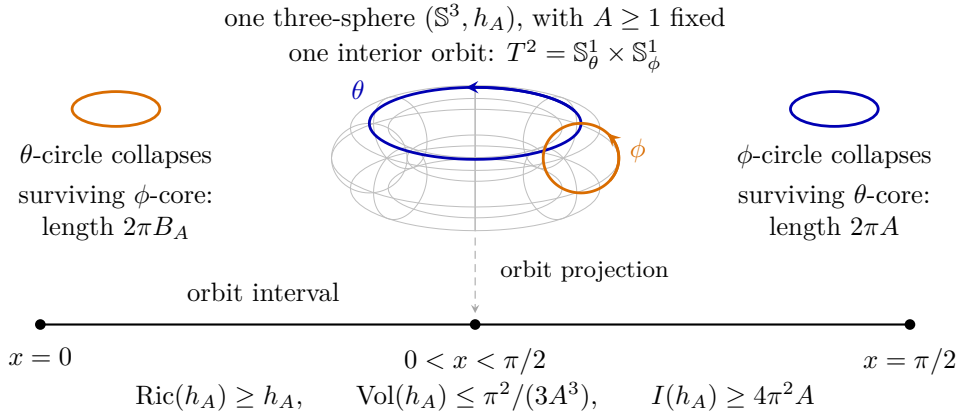
\begin{figure}[htbp]
\centering
\begin{tikzpicture}[x=1cm,y=1cm,>=stealth,
  every node/.style={font=\small},
  endpoint/.style={align=center,text width=3.8cm,font=\small}]
  \node at (6.2,4.05) {one three-sphere $(\Sph^3,h_A)$, with $A\geq1$ fixed};
  \node at (6.2,3.55) {one interior orbit: $T^2=\Sph^1_\theta\times\Sph^1_\phi$};
  \begin{scope}[shift={(6.2,2.2)}]
    
    \foreach \v in {0,60,120,180,240,300}{
      \draw[gray!50,line width=.28pt,domain=0:360,samples=65,variable=\u]
        plot ({(1.4+.5*cos(\v))*cos(\u)},
              {.34*(1.4+.5*cos(\v))*sin(\u)+.92*.5*sin(\v)});
    }
    \foreach \u in {0,45,90,135,180,225,270,315}{
      \draw[gray!50,line width=.28pt,domain=0:360,samples=65,variable=\v]
        plot ({(1.4+.5*cos(\v))*cos(\u)},
              {.34*(1.4+.5*cos(\v))*sin(\u)+.92*.5*sin(\v)});
    }
    
    \draw[blue!70!black,line width=1.05pt,domain=0:360,samples=81,variable=\u]
      plot ({1.4*cos(\u)},{.476*sin(\u)+.46});
    \draw[->,blue!70!black,line width=1.05pt,domain=20:95,samples=30,variable=\u]
      plot ({1.4*cos(\u)},{.476*sin(\u)+.46});
    \draw[orange!85!black,line width=1.05pt,domain=0:360,samples=81,variable=\v]
      plot ({1.4+.5*cos(\v)},{.46*sin(\v)});
    \draw[->,orange!85!black,line width=1.05pt,domain=-40:45,samples=30,variable=\v]
      plot ({1.4+.5*cos(\v)},{.46*sin(\v)});
    \node[blue!70!black] at (-1.55,.9) {$\theta$};
    \node[orange!85!black] at (2.15,.1) {$\phi$};
  \end{scope}
  \draw[->,densely dashed,gray!70] (6.2,1.15) -- (6.2,.15);
  \node[font=\footnotesize,anchor=west] at (6.4,.7) {orbit projection};
  \draw[orange!85!black,line width=1.05pt] (1.45,2.85) ellipse (.58 and .23);
  \draw[blue!70!black,line width=1.05pt] (10.95,2.85) ellipse (.58 and .23);
  \draw[thick] (.45,0) -- (11.95,0);
  \fill (.45,0) circle (2pt);
  \fill (6.2,0) circle (2pt);
  \fill (11.95,0) circle (2pt);
  \node[above=5pt] at (3.4,0) {orbit interval};
  \node[below=5pt] at (.45,0) {$x=0$};
  \node[below=5pt] at (6.2,0) {$0<x<\pi/2$};
  \node[below=5pt] at (11.95,0) {$x=\pi/2$};
  \node[endpoint] at (1.45,1.75)
    {$\theta$-circle collapses\\[3pt]
     surviving $\phi$-core:\\length $2\pi B_A$};
  \node[endpoint] at (10.95,1.75)
    {$\phi$-circle collapses\\[3pt]
     surviving $\theta$-core:\\length $2\pi A$};
  \node at (6.2,-.95)
    {$\Ric(h_A)\geq h_A$,\qquad
     $\Vol(h_A)\leq\pi^2/(3A^3)$,\qquad $I(h_A)\geq4\pi^2A$};
\end{tikzpicture}
\caption{The torus orbits of $(\Sph^3,h_A)$ for fixed $A\geq1$, shown schematically. At each endpoint, one circle collapses and the other survives. The coordinate $x$ runs along the orbit interval, with arclength element $q_A(x)\dd x$; it is distinct from the ambient radial coordinate $r$. Lemma~\ref{lem:riccislices} proves the displayed estimates.}
\label{fig:mechanism}
\end{figure}
\par
\subsection{Coefficients and smooth endpoints}
The parameter $A$ will tend to infinity; an auxiliary parameter $c$ connects the family to a round metric. Set $\varepsilon=1/4$ and let
\[
 \mathcal P=\bigl([0,1]\times\{\varepsilon\}\bigr)
       \cup\bigl(\{1\}\times[\varepsilon,\infty)\bigr).
\]
For $(c,A)\in\mathcal P$ and $0\leq a\leq A$, define
\begin{align}
 W(a)&=a^2+a^4, & w_c(a)&=cW(a),\\
 F_c(a)&=1-\int_0^a y e^{-2w_c(y)}\dd y,
 & \mathcal H_c(a)&=a^2e^{2w_c(a)},\\
 \eta_{c,A}&=\frac{F_c(A)}{\mathcal H_c(A)},
 & f_{c,A}(a)&=F_c(a)-\eta_{c,A}\mathcal H_c(a),\\
 B_{c,A}&=\frac{2}{-f_{c,A}'(A)}.\label{eq:coefficients}
\end{align}
The auxiliary coefficient $\mathcal H_c$ is distinct from the mean curvature $H$ introduced later. Unless stated otherwise, primes in this section denote $a$ derivatives. On $0<a<A$, define
\begin{equation}\label{eq:toricmetric}
 h_{c,A}=\frac{e^{-2w_c(a)}}{f_{c,A}(a)}\,\mathrm da^2
       +a^2\,\mathrm d\theta^2
       +B_{c,A}^2e^{-2w_c(a)}f_{c,A}(a)\,\mathrm d\phi^2,
 \qquad a=A\sin x.
\end{equation}

The identity $F_c^\prime=-ae^{-2w_c}$ gives $-e^{2w_c}F_c^\prime/a=1$ in the tangential Ricci calculation. The correction $-\eta_{c,A}\mathcal H_c$ enforces $f_{c,A}(A)=0$ and increases the radial Ricci expression, as \eqref{eq:DH} will show. The normalization $B_{c,A}$ makes the terminal arclength slope of $b$ equal to $-1$.

\begin{lemma}\label{lem:smooth}
For every $(c,A)\in\mathcal P$, \eqref{eq:toricmetric} extends to a smooth metric on the fixed $\Sph^3$. The extensions depend jointly smoothly on the parameters on each indicated parameter segment. Moreover, $h_{0,\varepsilon}=\varepsilon^2\gamma$.
\end{lemma}
\begin{proof}
For $c=1$, $W(a)\geq a^2$ gives $\int_0^a y e^{-2W(y)}\dd y\leq\int_0^\infty y e^{-2y^2}\dd y=1/4$, hence $3/4\leq F_1\leq1$. For $A=\varepsilon$ and $0\leq c\leq1$, we have $F_c(a)\geq1-a^2/2\geq31/32$. Thus $\eta_{c,A}>0$, $f(0)=1$, $f(A)=0$, and $f'<0$ on $(0,A]$; here and in the endpoint argument we suppress the parameters. Indeed, $f^\prime(a)=-ae^{-2w(a)}-2\eta a e^{2w(a)}(1+aw^\prime(a))<0$ for $a>0$. Consequently $f>0$ on $[0,A)$ and its final zero is simple. The arclength integrand is therefore $O((A-a)^{-1/2})$ at that endpoint and is integrable. Define
\[
 t(a)=\int_0^a\frac{e^{-w(y)}}{\sqrt{f(y)}}\dd y,
 \qquad L_{c,A}=t(A),\qquad b(t)=Be^{-w(a(t))}\sqrt{f(a(t))}.
\]
Then $h=\mathrm dt^2+a(t)^2\mathrm d\theta^2+b(t)^2\mathrm d\phi^2$. Near $a=0$, $w$ and $f$ are smooth even functions, with $w(0)=0$ and $f(0)=1$. Thus $t(a)$ and its inverse are smooth odd functions with derivative one at zero; $b(t)$ is smooth even and positive there. These parity and slope conditions give smoothness at the first core circle.

At the other endpoint write $f(a)=(A-a)G(a)$, where $G(A)=-f'(A)>0$, and put $z=\sqrt{A-a}$. The distance $\tau$ from that endpoint satisfies
\[
 \tau=\int_0^z\frac{2e^{-w(A-y^2)}}{\sqrt{G(A-y^2)}}\dd y.
\]
This is odd and smooth with positive derivative at zero. Hence $z(\tau)$ is odd and smooth, $a=A-z(\tau)^2$ is even, and
\[
 b=Be^{-w(A-z^2)}z\sqrt{G(A-z^2)}
\]
is odd with $\mathrm db/\mathrm d\tau=BG(A)/2=1$ at zero. This gives smoothness at the second core circle. The opposite circle collapses produce the stated $\Sph^3$.

To check smooth dependence on the parameters in fixed charts, return to the original coordinates. Put $v=a^2$ and define
\begin{align*}
 \widetilde F_c(v)&=1-\frac12\int_0^v e^{-2c(y+y^2)}\dd y,\\
 \Psi(c,A,v)&=\widetilde F_c(v)-\eta_{c,A}v e^{2c(v+v^2)},\\
 \mathcal K(c,A,v)&=-\int_0^1
  \partial_v\Psi(c,A,v+\lambda(A^2-v))\dd\lambda.
\end{align*}
The formulas for $\widetilde F_c$, $\Psi$, and $\mathcal K$ are smooth for $A>0$ and real $c$. For the parameters in $\mathcal P$, direct differentiation gives
\[
 \partial_v\Psi
 =-\tfrac12 e^{-2c(v+v^2)}
 -\eta_{c,A}e^{2c(v+v^2)}(1+2cv+4cv^2)<0
 \qquad (0\leq v\leq A^2).
\]
Since $\Psi(c,A,A^2)=0$, the fundamental theorem of calculus gives the positive smooth factorization
\[
 f_{c,A}(a)=(A^2-a^2)\mathcal K(c,A,a^2),\qquad
 \mathcal K(c,A,0)=A^{-2},\quad
 \mathcal K(c,A,A^2)=\frac{-f_{c,A}'(A)}{2A}.
\]
Smoothness on a parameter segment means that the tensor coefficients extend smoothly to an open neighborhood of each parameter value. To verify this at $(c_0,A_0)\in\mathcal P$, note that $\mathcal K(c_0,A_0,A_0^2\xi)$ has a positive minimum on $0\leq\xi\leq1$. By continuity and compactness, it stays positive for all such $\xi$ on a sufficiently small open neighborhood of $(c_0,A_0)$. On that neighborhood, $B=1/(A\mathcal K(c,A,A^2))$, the positive square roots, and their reciprocals are smooth. This argument applies also at $c_0=0$ and at the meeting point of the two parameter segments.

This factorization cancels the apparent singularity in the coefficient of $\mathrm dx^2$. Suppressing the parameters, we obtain the fixed-coordinate coefficients
\[
 q(x)^2=\frac{e^{-2w(A\sin x)}}{\mathcal K(c,A,A^2\sin^2x)},\qquad
 b(x)^2=B^2 e^{-2w(A\sin x)}A^2\cos^2x\,
                   \mathcal K(c,A,A^2\sin^2x).
\]
Take $q>0$ and $b=BA\cos x\,e^{-w(A\sin x)}\sqrt{\mathcal K(c,A,A^2\sin^2x)}$ on the orbit interval. These formulas give the full parity at both ends, jointly in the parameters. At the first end $a_x(0)=A=q(0)$. At the second, the ratio of $-b_x(\pi/2)$ to $q(\pi/2)$ is $BA\mathcal K(c,A,A^2)=B(-f'(A))/2=1$. Thus the collapsing radii have unit derivative with respect to inward arclength. At $(c,A)=(0,\varepsilon)$, the formulas reduce to $f=1-a^2/\varepsilon^2$ and $B=\varepsilon$, giving $h_{0,\varepsilon}=\varepsilon^2\gamma$.
At the first core, introduce $X=x\cos\theta$, $Y=x\sin\theta$, together with a local angular coordinate $\phi$ on the surviving circle. These are coordinates for the fixed sphere: the first complex component gives the normal disk coordinates $(\sin x/x)(X,Y)$, whose radial factor is smooth in $x^2$ and equals one at zero. The coordinate change is therefore a local diffeomorphism at the core. With $u(x)=a(x)/x$, the normal part is
\[
 q^2\,\mathrm dx^2+a^2\,\mathrm d\theta^2
 =u^2(\mathrm dX^2+\mathrm dY^2)
 +\frac{q^2-u^2}{x^2}(X\,\mathrm dX+Y\,\mathrm dY)^2.
\]
The displayed coefficient formulas show that $q,u,b$ are smooth functions of $x^2$ and the parameters, with $u(0)=q(0)>0$. If $q^2-u^2=N(x^2)$, then $N(0)=0$ and $N(v)/v=\int_0^1N'(\lambda v)\dd\lambda$ is smooth, also in the parameters. Every Cartesian coefficient is therefore smooth in $X,Y$, since $x^2=X^2+Y^2$, and the tensor is positive definite at the core. Including the surviving circle, the metric at this core is $A^2(\mathrm dX^2+\mathrm dY^2)+B^2\mathrm d\phi^2$.

At the second core, put $y=\pi/2-x$, $U=y\cos\phi$, $V=y\sin\phi$, and $v_*=b/y$. The same radial coordinate change, now with factor $\sin y/y$, identifies $(U,V,\theta)$ with a smooth chart of the fixed sphere. Write $K_A=\mathcal K(c,A,A^2)$. Since $BAK_A=1$,
\[
 v_*(0)=BAe^{-w(A)}\sqrt{K_A}
       =\frac{e^{-w(A)}}{\sqrt{K_A}}=q(\pi/2).
\]
The functions $q,v_*,a$ are smooth in $y^2$ and the parameters. Moreover,
\[
 q^2\mathrm dy^2+b^2\mathrm d\phi^2
 =v_*^2(\mathrm dU^2+\mathrm dV^2)
  +\frac{q^2-v_*^2}{y^2}(U\mathrm dU+V\mathrm dV)^2.
\]
The same integral division formula makes the last coefficient smooth. At this core the full metric is $q(\pi/2)^2(\mathrm dU^2+\mathrm dV^2)+A^2\mathrm d\theta^2$, which is positive definite. These two local extensions agree with the interior tensor on their overlaps, proving the claimed joint smooth extension.
\end{proof}

\subsection{Ricci curvature and natural integrals}
\begin{lemma}\label{lem:riccislices}
Set $\kappa=23/32$. On the whole parameter set $\mathcal P$, $\Ric(h_{c,A})\geq\kappa h_{c,A}$. On the segment $c=1$ the stronger inequality $\Ric(h_{1,A})\geq h_{1,A}$ holds. For $A\geq1$, writing $h_A=h_{1,A}$,
\begin{equation}\label{eq:massvolume}
 \Vol(h_A)\leq\frac{\pi^2}{3A^3},\qquad
 I(h_A)=\int_{\Sph^3}\Scal_{h_A}\dd V_{h_A}\geq4\pi^2 A.
\end{equation}
\end{lemma}
\begin{proof}
Fix $(c,A)\in\mathcal P$ and suppress these parameters. Use the arclength coordinate and coefficient bounds from Lemma~\ref{lem:smooth}. All $a$ derivatives hold $c,A$, and hence $\eta_{c,A}$ and $B$, fixed.

For a doubly warped metric $\mathrm dt^2+a^2\mathrm d\theta^2+b^2\mathrm d\phi^2$, use the orthonormal frame $e_0=\partial_t$, $e_1=a^{-1}\partial_\theta$, $e_2=b^{-1}\partial_\phi$ in the interior. Dots denote $t$ derivatives. Its nonzero covariant derivatives are
\[
 \nabla_{e_1}e_0=\frac{\dot a}{a}e_1,\quad
 \nabla_{e_1}e_1=-\frac{\dot a}{a}e_0,\quad
 \nabla_{e_2}e_0=\frac{\dot b}{b}e_2,\quad
 \nabla_{e_2}e_2=-\frac{\dot b}{b}e_0.
\]
For example, $[e_0,e_1]=-(\dot a/a)e_1$ and the convention $R(X,Y)Z=\nabla_X\nabla_YZ-\nabla_Y\nabla_XZ-\nabla_{[X,Y]}Z$ give $\langle R(e_0,e_1)e_1,e_0\rangle=-\ddot a/a$. The other two planes give
\[
 K_{t\theta}=-\frac{\ddot a}{a},\qquad
 K_{t\phi}=-\frac{\ddot b}{b},\qquad
 K_{\theta\phi}=-\frac{\dot a\dot b}{ab}.
\]
The connection formulas also give zero off-diagonal Ricci entries. Its eigenvalues are $\lambda_t=K_{t\theta}+K_{t\phi}$, $\lambda_\theta=K_{t\theta}+K_{\theta\phi}$ and $\lambda_\phi=K_{t\phi}+K_{\theta\phi}$. For \eqref{eq:toricmetric}, $\dot a=e^w\sqrt f$, $\dot a b=Bf$, and $\dot b=B(f'/2-w'f)$. Define the linear differential expression
\[
 D_w[f]=f\left(w''-\frac{w'}a\right)+w'f'
                 -\frac12\left(f''+\frac{f'}a\right).
\]
The second arclength derivatives and the remaining sectional curvature are
\[
 \ddot a=e^{2w}(w'f+\tfrac12f'),\qquad
 \frac{\ddot b}{b}=e^{2w}(\tfrac12f''-w''f-w'f'),\qquad
 K_{\theta\phi}=\frac{e^{2w}}a(w'f-\tfrac12f').
\]
Differentiating with $\mathrm d/\mathrm dt=e^w\sqrt f\,\mathrm d/\mathrm da$ gives all three Ricci eigenvalues:
\begin{equation}\label{eq:slicericci}
 \lambda_t=e^{2w}D_w[f],\qquad
 \lambda_\theta=-e^{2w}\frac{f'}a,\qquad
 \lambda_\phi=\lambda_t+2e^{2w}f\frac{w'}a.
\end{equation}
To evaluate this expression, use $F_c'=-ae^{-2w_c}$, $F_c''=(-1+2aw_c')e^{-2w_c}$, and $\mathcal H_c'/\mathcal H_c=2/a+2w_c'$. In particular, $\mathcal H_c''/\mathcal H_c=(2/a+2w_c')^2-2/a^2+2w_c''$. Substitution, using $w_c''-w_c'/a=8ca^2$, yields
\begin{align}
 e^{2w_c}D_{w_c}[F_c]
   &=1+8ca^2F_ce^{2w_c}-2caW',\\
 D_{w_c}[\mathcal H_c]&=-\mathcal H_c\left(\frac2{a^2}+\frac{4w_c'}a\right).
 \label{eq:DH}
\end{align}
The correction contributes a positive term to the radial eigenvalue:
\[
 \lambda_t=1+8ca^2F_ce^{2w_c}-2caW'
            +\eta_{c,A}e^{4w_c}(2+4aw_c').
\]
Also, $\lambda_\theta=1+\eta_{c,A}e^{2w_c}\mathcal H_c'/a\geq1$. Since $w_c'\geq0$, $\lambda_\phi\geq\lambda_t$. At $A=\varepsilon$,
\[
 \lambda_t\geq1-2caW'\geq1-4\varepsilon^2-8\varepsilon^4
       =\frac{23}{32}.
\]
For $c=1$, use $F_1e^{2W}\geq\frac34(1+2a^2)\geq\frac12+a^2$ to obtain
\[
 e^{2W}D_W[F_1]=1+8a^2F_1e^{2W}-4a^2-8a^4\geq1.
\]
The inequalities extend to both core circles by smoothness.

The two angles have period $2\pi$, so their integration contributes $4\pi^2$. The two core circles have zero three-dimensional volume. Thus integrals over the whole sphere can be computed on the interior orbit interval with natural density $ab\dd t\dd\theta\dd\phi$.

For the volume and scalar integrals, set $c=1$ and suppress this parameter from the notation. Because $\mathrm dt=e^{-W}f^{-1/2}\mathrm da$ and $b=Be^{-W}f^{1/2}$, the factors involving $f$ cancel in the natural volume form. Thus
\begin{equation}\label{eq:naturalvolume}
 V(A)=4\pi^2\int_0^{L_{1,A}}ab\dd t
     =4\pi^2 B_A\int_0^A a e^{-2W(a)}\dd a\leq\pi^2 B_A.
\end{equation}
The first derivative formulas give all endpoint data, with $t$ increasing from the first core to the second:
\[
 (a,b,\dot a,\dot b)|_{t=0}=(0,B_A,1,0),\qquad
 (a,b,\dot a,\dot b)|_{t=L_{1,A}}=(A,0,0,-1).
\]
For $A\geq1$,
\[
 -f_A'(A)=Ae^{-2W(A)}+F_1(A)\left(\frac2A+2W'(A)\right)
     \geq6A^3,
\]
Thus $B_A\leq1/(3A^3)$. For the scalar integral, the tangential Ricci eigenvalues satisfy $\lambda_\theta ab=-(\dot a b)^{\cdot}$ and $\lambda_\phi ab=-(a\dot b)^{\cdot}$, whence
\[
 \int\lambda_\theta\dd V_{h_A}=4\pi^2B_A,\qquad
 \int\lambda_\phi\dd V_{h_A}=4\pi^2A.
\]
Indeed, integrating first over a truncated orbit interval and passing to the endpoints gives $-4\pi^2[\dot a b]_0^{L_{1,A}}=4\pi^2B_A$ and $-4\pi^2[a\dot b]_0^{L_{1,A}}=4\pi^2A$. Boundedness of the Ricci tensor on the compact sphere justifies passage to the endpoints.
Combining the two flux identities with the trace relation for scalar curvature gives
\[
 I(h_A)=4\pi^2(A+B_A)
       +\int_{\{0<x<\pi/2\}}\lambda_t\dd V_{h_A}.
\]
Since $\lambda_t\geq1$ on the principal orbits, we obtain $I(h_A)\geq4\pi^2(A+B_A)+V(A)$ and hence \eqref{eq:massvolume}.
\relax{}
\end{proof}

\section{A path of cross-sections}\label{sec:path}

We join the slices of Lemma~\ref{lem:riccislices} to the round metric. To control the resulting terms in the ambient Ricci tensor, we make the volume contraction spatially constant and reparameterize the path to bound its deformation endomorphism, parameter derivative, and spatial divergence.

\relax{}
Four parameters appear, with distinct roles: $A$ labels the explicit toric slices, $z$ joins those slices into a raw path, $s$ is a slowed parameter on the same path, and $r$ is the distance variable of the final manifold. The successive changes are
\[
 h_A\ \longrightarrow\ k_z\ \longrightarrow\ \bar k_s
 \ \longrightarrow\ \rho(r)^2\bar k_{s(r)}.
\]
\relax{}

\subsection{A path with nonincreasing volume}

We first show that the explicit family $h_{1,A}$ already has decreasing total volume for $A\geq1$. The finite connecting segment will be treated in Proposition~\ref{prop:path}.

\begin{lemma}\label{lem:decreasing}
The natural volume $V(A)=\Vol(h_{1,A})$ is strictly decreasing for $A\geq1$.
\end{lemma}
\begin{proof}
Put
\[
 J(A)=\int_0^A a e^{-2W(a)}\dd a,\quad
 E(A)=\frac2A+4A+8A^3,\quad
 D(A)=J'(A)+(1-J(A))E(A).
\]
Then $D(A)=-\left.\partial_a f_{1,A}(a)\right|_{a=A}>0$ and $B_{1,A}=2/D$, so \eqref{eq:naturalvolume} gives $V=8\pi^2J/D$. In the remaining formulas of this proof, primes denote $A$ derivatives, and $F=1-J$. Since $F'=-J'$, $J''=J'(1/A-2W')$, and $E=2/A+2W'$, we have $D'=J''-J'E+FE'$. Thus
\begin{align}
 D'&=FE'-J'(1/A+4W'),\\
 JD'-J'D&=JFE'-JJ'(1/A+4W')-(J')^2-J'FE. \label{eq:volderiv}
\end{align}
For $A\geq1$ we have $3/64\leq J\leq1/4$, $3/4\leq F\leq1$, and $E'\geq26$. The lower bound for $J$ follows by integrating on $[0,1/2]$, where $W\leq5/16$ and $e^{-5/8}\geq1-5/8=3/8$. The upper bound follows from $J\leq\int_0^\infty a e^{-2a^2}\dd a=1/4$. The functions $A^{2k}e^{-2W(A)}$, $k=0,1,2$, decrease on $[1,\infty)$, since their logarithmic derivatives are $2k/A-4A-8A^3<0$. It follows that
\[
 J'(1/A+4W')\leq25e^{-4},\qquad
 J'E\leq14e^{-4},\qquad (J')^2\leq e^{-8}.
\]
For the first two bounds, the polynomials multiplying $e^{-2W}$ are respectively $1+8A^2+16A^4$ and $2+4A^2+8A^4$; their coefficients sum to $25$ and $14$. For the third, $A^2e^{-4W}=(A^2e^{-2W})e^{-2W}\leq e^{-8}$. Also $E'=-2/A^2+4+24A^2\geq26$.
The first five terms of the exponential series give $e^4\geq103/3>32$. Using this in \eqref{eq:volderiv}, we obtain
\[
 JD'-J'D\geq
 \frac3{64}\frac34\,26-\frac{25/4+14}{32}-\frac1{1024}
 =\frac{287}{1024}>0.
\]
Since $V'=8\pi^2(J'D-JD')/D^2$, the assertion follows.
\end{proof}

\begin{proposition}\label{prop:path}
There is a jointly smooth family of metrics $k_z$, $z\geq0$, on the fixed $\Sph^3$ such that
\begin{align}
 k_0&=\gamma,\qquad \Ric(k_z)\geq\kappa k_z,\qquad \kappa=23/32,\label{eq:pathric}\\
 \dd V_{k_z}&=v(z)\dd V_\gamma,\qquad v(0)=1,\quad v>0,\quad v'\leq0,\label{eq:pathdensity}\\
 I(k_z)&=\int_{\Sph^3}\Scal_{k_z}\dd V_{k_z}
       \longrightarrow\infty\quad (z\longrightarrow\infty).\label{eq:pathmass}
\end{align}
\end{proposition}
We call \eqref{eq:pathdensity} the common volume gauge: the volume forms differ only by the spatially constant factor $v(z)$.
For smooth positive volume forms of equal total mass on a closed connected manifold, Moser~\cite[pp.~286--287]{Moser1965} gives a diffeomorphism relating the forms. We give an explicit change of orbit coordinate after normalizing the total mass. The condition $v'\leq0$ will give the logarithmic volume derivative the required sign in the radial Ricci estimate.
\begin{proof}
\emph{The connecting path.} Fix the smooth nondecreasing function $\eta:\R\to[0,1]$ given by
\[
 \eta(y)=\frac{\vartheta(y)}{\vartheta(y)+\vartheta(1-y)},\qquad
 \vartheta(y)=
 \begin{cases}e^{-1/y},&y>0,\\0,&y\leq0.\end{cases}
\]
It is zero for $y\leq0$, one for $y\geq1$, and all its positive-order derivatives vanish at the two endpoints. First define a raw path on the fixed sphere:
\begin{equation}\label{eq:rawpath}
 \widehat h_z=
 \begin{cases}
 [1-(1-\varepsilon)\eta(z)]^2\gamma,&0\leq z\leq1,\\
 h_{\eta(z-1),\varepsilon},&1\leq z\leq2,\\
 h_{1,\varepsilon+(1-\varepsilon)\eta(z-2)},&2\leq z\leq3,\\
 h_{1,1+(z-3)\eta(z-3)},&z\geq3.
 \end{cases}
\end{equation}
At $z=1,2,3$, the common values in \eqref{eq:rawpath} are $\varepsilon^2\gamma$, $h_{1,\varepsilon}$, and $h_{1,1}$, respectively. By Lemma~\ref{lem:smooth}, on both sides of each join every positive-order parameter derivative of $\widehat h_z$ vanishes: the varying parameter differs from its endpoint value by a flat function. At $z=4$, $A(z)$ agrees to all orders with $z-2$, because $\eta(z-3)-1$ is flat there.

Lemma~\ref{lem:riccislices} gives $\Ric(\widehat h_z)\geq\kappa\widehat h_z$; the round segment has Ricci eigenvalues at least two. Its volume $U(z)=\Vol(\widehat h_z)$ is positive and smooth. Lemma~\ref{lem:decreasing} gives $U'\leq0$ for $z\geq3$.
Indeed, on this last segment $A(z)=1+(z-3)\eta(z-3)\geq1$ and $A'(z)=\eta(z-3)+(z-3)\eta'(z-3)\geq0$, so $U'=V'(A)A'\leq0$.

\par\smallskip\noindent\emph{Nonincreasing total volume.} To correct the volume on the finite connecting segment, set
\[
 M=1+\frac13\max_{[0,3]}\left|\frac{U'}U\right|,\qquad
 \chi(z)=1-\eta(z-3),\qquad
 \ell(z)=\exp\left(-M\int_0^z\chi(y)\dd y\right).
\]
The maximum is finite. Moreover, $\ell(0)=1$, $0<\ell\leq1$, and $\ell$ is a fixed positive constant $\ell_\infty$ for $z\geq4$. Define $g_z=\ell(z)^2\widehat h_z$. Then
\[
 (\log\Vol(g_z))'=\frac{U'}U-3M\chi\leq0.
\]
For $z\leq3$ this follows from the choice of $M$, and for $z\geq3$ from $U'\leq0$. Since $\ell$ is constant on each spatial slice, the following inequalities hold for $(0,2)$ tensors: $\Ric(\ell^2\widehat h_z)=\Ric(\widehat h_z)\geq\kappa\widehat h_z\geq\kappa\ell^2\widehat h_z$, because $0<\ell\leq1$. In dimension three, total scalar curvature scales by one power of $\ell$, so $I(g_z)=\ell(z)I(\widehat h_z)$. Since $\ell=\ell_\infty>0$ for $z\geq4$, \eqref{eq:massvolume} and $A(z)=z-2$ there imply $I(g_z)\to\infty$.

\par\smallskip\noindent\emph{A common volume form.} It remains to arrange the density pointwise. In the fixed toric coordinate write
\[
 g_z=q_z(x)^2\mathrm dx^2+a_z(x)^2\mathrm d\theta^2
                         +b_z(x)^2\mathrm d\phi^2.
\]
Choose the positive square roots in the interior. The function $q_z$ extends positively to both endpoints; $a_z$ has a simple zero only at $x=0$, and $b_z$ only at $x=\pi/2$. These facts follow from the explicit round and toric formulas and are preserved by the positive scaling $\ell$. Define a normalized cumulative volume and its associated change of coordinate by
\begin{align}
 Z_z&=\int_0^{\pi/2}q_za_zb_z\dd x=\frac{\Vol(g_z)}{4\pi^2},&
 C_z(x)&=\frac1{Z_z}\int_0^x q_z(y)a_z(y)b_z(y)\dd y,\\
 T_z(x)&=C_z^{-1}(\sin^2x),&
 \Phi_z(x,\theta,\phi)&=(T_z(x),\theta,\phi).
\end{align}
We verify that this coordinate formula defines a smooth diffeomorphism even at the collapsed circles. The coefficient formulas of Lemma~\ref{lem:smooth} give
\[
 q_z(x)a_z(x)b_z(x)=\sin x\cos x\,R_z(\sin^2x),
 \qquad R_z(u)>0\quad(0\leq u\leq1),
\]
where $R_z(u)$ is jointly smooth up to both endpoints and across the path joins. On a round piece it is $[\ell(z)(1-(1-\varepsilon)\eta(z))]^3$, the cube of the radius of $g_z$; on a scaled toric piece it is $\ell^3BA^2e^{-2w(A\sqrt u)}$, a smooth function of $u$ because $w(a)=c(a^2+a^4)$.
The unscaled density factors have matching values and zero positive-order parameter derivatives at $z=1,2,3$; at the first join both equal $\varepsilon^3$. Multiplication by the common smooth factor $\ell(z)^3$ preserves matching derivatives. The substitution $u=\sin^2x$ also gives $Z_z=\frac12\int_0^1R_z(u)\dd u$.
\par\noindent\begin{minipage}{\linewidth}
In the variable $u$, define

\[
 \mathcal C_z(u)=\frac{\int_0^u R_z(v)\dd v}{\int_0^1 R_z(v)\dd v},
 \qquad \mathcal J_z=\mathcal C_z^{-1}.
\]\end{minipage}\par
Fix a finite $z_0$. The explicit formulas extend $R_z(u)$ smoothly to a neighborhood of $\{z_0\}\times[0,1]$; after shrinking it, positivity on the compact interval keeps this extension positive. The map $(z,u)\mapsto(z,\mathcal C_z(u))$ has Jacobian determinant $\mathcal C_z^\prime(u)>0$ there. Its local inverses therefore exist also at the interval endpoints and agree with the unique increasing inverse on $[0,1]$.

These inverses give a jointly smooth family $\mathcal J_z$ up to $u=0,1$, with $\mathcal J_z(0)=0$, $\mathcal J_z(1)=1$ and $\mathcal J_z'>0$. Moreover, $C_z(x)=\mathcal C_z(\sin^2x)$. The construction is local in $z$ and therefore needs no uniform derivative bound as $z\to\infty$. In the fixed model $\Sph^3\subset\mathbb C^2$, the proposed map is precisely
\[
 \Phi_z(\zeta_1,\zeta_2)=
 \left(\sqrt{\frac{\mathcal J_z(u)}u}\,\zeta_1,
       \sqrt{\frac{1-\mathcal J_z(u)}{1-u}}\,\zeta_2\right),
 \qquad u=|\zeta_1|^2.
\]
The apparent endpoint quotients extend positively and smoothly: they equal $\int_0^1\mathcal J_z'(\lambda u)\dd\lambda$ and $\int_0^1\mathcal J_z'(u+\lambda(1-u))\dd\lambda$, respectively. Hence the displayed formula is jointly smooth on the whole sphere. The squared moduli of its two components sum to one, and replacing $\mathcal J_z$ by $\mathcal C_z$ gives its smooth inverse. 

Set $k_z=\Phi_z^*g_z$. Differentiating $C_z(T_z(x))=\sin^2x$ gives
\[
 q_z(T_z)a_z(T_z)b_z(T_z)T_z'=2Z_z\sin x\cos x.
\]
Since $T_z'>0$, the Jacobian identity gives $\Phi_z^*\dd V_{g_z}=2Z_z\dd V_\gamma$ on the interior. Both sides are smooth global forms, so they agree also at the two cores. Thus \eqref{eq:pathdensity} holds with $v(z)=2Z_z=\Vol(g_z)/(2\pi^2)>0$, and $v^\prime/v=(\log\Vol(g_z))^\prime\leq0$. Since $g_0=\gamma$, one has $Z_0=1/2$ and $v(0)=1$. At $z=0$, $C_0(x)=\sin^2x$, so $\Phi_0=\Id$ and $k_0=\gamma$. Each $\Phi_z$ is a diffeomorphism of the fixed sphere. Naturality of curvature under pullback and the change-of-variables formula~\cite[Proposition 7.6 and pp.~30--32, 403--406]{Lee2018} therefore preserve both the tensor inequality and the total scalar integral under $k_z=\Phi_z^*g_z$.
\end{proof}

The constructed path also has vanishing volume. For $z\geq4$, we have $\widehat h_z=h_{1,z-2}$ and $\ell(z)=\ell_\infty>0$. Since pullback preserves total volume, \eqref{eq:massvolume} gives
\[
 \Vol(k_z)=\ell_\infty^3\Vol(h_{1,z-2})
 \leq\frac{\ell_\infty^3\pi^2}{3(z-2)^3}\longrightarrow0.
\]
In the common volume gauge this is $v(z)=\Vol(k_z)/(2\pi^2)\to0$.
\par
The derivative estimates below apply to this pulled-back path, including the parameter dependence of $\Phi_z$.

\subsection{Uniform control of parameter derivatives}\label{sec:clock}

Joint smoothness gives derivative bounds on compact parameter intervals. By slowing each successive interval, we obtain uniform bounds while keeping the parameter map proper.
\par
For a continuous map $\Theta:[0,\infty)\to[0,\infty)$, \emph{proper} means that inverse images of compact sets are compact; here this is equivalent to $\Theta(s)\to\infty$. Our parameter maps may have constant intervals and need not be diffeomorphisms. All tensor parameter derivatives below are taken on the fixed bundle $T\Sph^3$. For an endomorphism $E$ and a covector $\xi$, our norm conventions are
\[
 |E|_k^2=\sum_{i=1}^3 k(Ee_i,Ee_i),\qquad
 |\xi|_k^2=\sum_{i=1}^3\xi(e_i)^2,
 \qquad k(e_i,e_j)=\delta_{ij}.
\]
These are independent of the chosen $k$-orthonormal basis. We also use the operator norm $\|E\|_{\rm op}=\sup_{|v|_k=1}|Ev|_k$, which is bounded above by $|E|_k$. The current slice metric always determines these norms. For an endomorphism $E$, we write $(\divg E)_i=\nabla_jE^j{}_i$ for its divergence covector, where $\nabla$ is the Levi--Civita connection of the current metric $k$.

\begin{lemma}\label{lem:slow}
There is a smooth, nondecreasing, proper change of parameter for the path in Proposition~\ref{prop:path}. The resulting path $\bar k_s$, $s\geq0$, satisfies $\bar k_s=\gamma$ for $0\leq s\leq1$. It preserves the Ricci bound, density properties, and divergence of the scalar integral stated there. Define the endomorphism and its trace
\[
 \mathcal A=\frac12\bar k_s^{-1}\partial_s\bar k_s,\qquad
 \beta=\tr\mathcal A.
\]
Then, in the current metric $\bar k_s$,
\begin{equation}\label{eq:slowbounds}
 |\mathcal A|\leq1,\qquad |\partial_s\mathcal A|\leq1,\qquad
 |\divg\mathcal A|\leq1,\qquad
 -2\leq\beta\leq0,\qquad |\partial_s\beta|\leq2.
\end{equation}
Here endomorphisms have the Hilbert--Schmidt norm, divergence is a covector, and $\beta$ is constant on each spatial slice.
\end{lemma}
\begin{proof}
For the original parameter $z$ set $\mathcal B_z=\frac12 k_z^{-1}\partial_z k_z$. For each integer $j\geq0$, choose $C_j\geq1$ bounding the norms of $\mathcal B_z$, $\partial_z\mathcal B_z$, and $\divg\mathcal B_z$ on $[j,j+1]\times\Sph^3$. These constants are finite by joint smoothness and compactness. Let $M_1=\sup|\eta'|$ and $M_2=\sup|\eta''|$, where $\eta$ is the step function used in \eqref{eq:rawpath}, and put
\[
 L_j=\max\{1,M_1C_j,\sqrt{(M_2+M_1^2)C_j}\},\qquad
 U_0=1,\quad U_{j+1}=U_j+L_j.
\]
Define $\Theta(s)=0$ for $s\leq1$, and on each interval $[U_j,U_{j+1}]$ set
\[
 \Theta(s)=j+\eta\left(\frac{s-U_j}{L_j}\right).
\]
At $U_{j+1}$ both formulas have value $j+1$, and every positive-order derivative vanishes on both sides by flatness of $\eta$; the join at $U_0=1$ has the same property. Each $L_j$ is finite and $L_j\geq1$, so $U_j\to\infty$ and the intervals cover $[1,\infty)$ without finite accumulation. For every nonnegative integer $N$, $s\geq U_N$ implies $\Theta(s)\geq N$; hence $\Theta$ is proper. Each interval maps onto $[j,j+1]$, so every finite original parameter is attained in finite time. Set $\bar k_s=k_{\Theta(s)}$. Its scalar integral is $I(\bar k_s)=I(k_{\Theta(s)})\to\infty$ by \eqref{eq:pathmass}, and its remaining geometric properties follow from Proposition~\ref{prop:path}. The chain rule gives
\[
 \mathcal A=\Theta'\mathcal B,\qquad
 \partial_s\mathcal A=\Theta''\mathcal B+(\Theta')^2\partial_z\mathcal B,
 \qquad \divg\mathcal A=\Theta'\divg\mathcal B.
\]
The choice of $L_j$ gives
\[
 |\mathcal A|,\ |\divg\mathcal A|
 \leq\frac{M_1C_j}{L_j}\leq1,\qquad
 |\partial_s\mathcal A|
 \leq\frac{(M_2+M_1^2)C_j}{L_j^2}\leq1.
\]
All these bounds are taken in the current slice metric. Differentiating $\bar k_s^{-1}\bar k_s=\Id$ gives $\partial_s(\bar k_s^{-1})=-\bar k_s^{-1}(\partial_s\bar k_s)\bar k_s^{-1}$. Thus
\[
 \partial_s\mathcal A=\tfrac12\bar k_s^{-1}\partial_s^2\bar k_s
                         -2\mathcal A^2
\]
is a sum of two self-adjoint endomorphisms in the current metric. For the trace, the local volume formula $\dd V_k=\sqrt{\det k}\,\mathrm dx^1\mathrm dx^2\mathrm dx^3$ yields
\[
 \partial_s\log\sqrt{\det\bar k_s}
   =\tfrac12\tr(\bar k_s^{-1}\partial_s\bar k_s)=\beta.
\]
Write $\bar v(s)=v(\Theta(s))$. The common-density identity therefore gives $\beta=(\log\bar v)'$, independent of the spatial point and nonpositive since $v'\leq0$ and $\Theta'\geq0$. Trace commutes with differentiation on the fixed bundle, so $\partial_s\beta=\tr(\partial_s\mathcal A)$. Applying Cauchy--Schwarz to the three eigenvalues of each self-adjoint endomorphism now gives $|\beta|,|\partial_s\beta|\leq\sqrt3<2$.
\end{proof}

In what follows, write $\bar v(s)=v(\Theta(s))$ and $S(s)=I(\bar k_s)$. The resulting path satisfies $\dd V_{\bar k_s}=\bar v(s)\dd V_\gamma$ and $\bar v'\leq0$. The volume limit just proved for $k_z$ and properness of $\Theta$ give $\bar v(s)\to0$, while $S(s)\to\infty$.

\section{The radial end and its Ricci tensor}\label{sec:ambient}

We now place the path from Lemma~\ref{lem:slow} on a radial end. The common volume gauge and the derivative bounds will allow us to control all three blocks of its Ricci tensor.

Fix the constants
\begin{equation}\label{eq:constants}
 K=100,\qquad T=1000,\qquad \alpha=\frac1{100},\qquad L=e^T.
\end{equation}
For $r\geq L$ set
\begin{equation}\label{eq:tail}
 t=\log r,\quad
 \rho(r)=\alpha r e^{K/t},\quad s(r)=\log(t/T),\quad
 G=\mathrm dr^2+\rho(r)^2\bar k_{s(r)}.
\end{equation}
Here $s(L)=0$ and $s(r)\to\infty$. The metric is smooth and positive definite for $r\geq L$. On $L\leq r\leq e^{eT}$, the path is constant: $\bar k_{s(r)}=\gamma$. Extending $\bar k_s$ by $\gamma$ for $s<0$ therefore gives a smooth round collar across $r=L$.
\par
For $r\geq L$, write $\Sigma_r=\{r\}\times\Sph^3$ and $h_r=\rho(r)^2\bar k_{s(r)}$ for a radial slice and its induced metric. Figure~\ref{fig:radial-shell} previews the completed metric of Section~\ref{sec:completion} and the shell used in its scalar integral estimate.
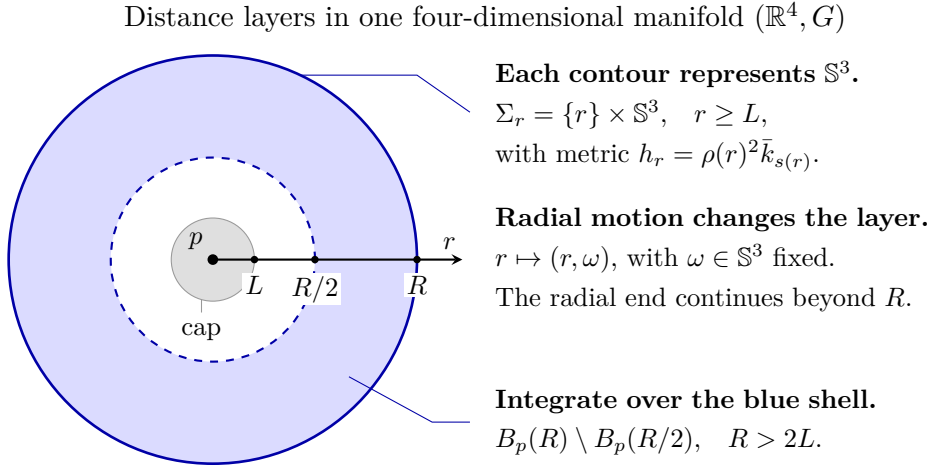
\begin{figure}[htbp]
\centering
\begin{tikzpicture}[x=1cm,y=1cm,>=stealth,
  every node/.style={font=\small},
  explanation/.style={anchor=west,align=left,text width=6.4cm}]
  \node[font=\normalsize] at (6.6,6.55)
    {Distance layers in one four-dimensional manifold $(\mathbb R^4,G)$};
  \begin{scope}[shift={(3,3.35)}]
    \fill[blue!14] (0,0) circle (2.7);
    \fill[white] (0,0) circle (1.35);
    \fill[gray!25] (0,0) circle (.55);
    \draw[gray!80] (0,0) circle (.55);
    \draw[blue!65!black,dashed,line width=.8pt] (0,0) circle (1.35);
    \draw[blue!65!black,line width=1pt] (0,0) circle (2.7);
    \draw[->,thick] (0,0) -- (3.3,0);
    \fill (0,0) circle (2pt);
    \node[above left] at (0,0) {$p$};
    \foreach \a/\txt in {.55/L,1.35/{R/2},2.7/R}{
      \fill (\a,0) circle (1.4pt);
      \node[below=5pt,fill=white,inner sep=1pt] at (\a,0) {$\txt$};
    }
    \node[above] at (3.13,0) {$r$};
    \draw[gray!80] (-.15,-.53) -- (-.15,-.7);
    \node at (-.15,-.95) {cap};
  \end{scope}
  \draw[blue!65!black] (4.16,5.79) -- (5.7,5.79) -- (6.4,5.3);
  \node[explanation] at (6.6,5.25)
    {\textbf{Each contour represents $\Sph^3$.}\\[3pt]
     $\Sigma_r=\{r\}\times\Sph^3$,\quad $r\geq L$,\\[3pt]
     with metric $h_r=\rho(r)^2\bar k_{s(r)}$.};
  \node[explanation] at (6.6,3.35)
    {\textbf{Radial motion changes the layer.}\\[3pt]
     $r\mapsto(r,\omega)$, with $\omega\in\Sph^3$ fixed.\\[3pt]
     The radial end continues beyond $R$.};
  \draw[blue!65!black] (4.8,1.8) -- (5.9,1.2) -- (6.4,1.2);
  \node[explanation] at (6.6,1.2)
    {\textbf{Integrate over the blue shell.}\\[3pt]
     $B_p(R)\setminus B_p(R/2)$,\quad $R>2L$.};
  \node at (6.6,-.2)
    {$\mathcal Q_{G,p}(R)\geq\dfrac{3\alpha}{16}
      \inf_{r\in[R/2,R]}S(s(r))\longrightarrow+\infty
      \quad(R\to\infty)$};
  \node[font=\footnotesize] at (6.6,-.85)
    {$S(s)=I(\bar k_s)$\qquad\textbullet\qquad
      distance-layer schematic; contours do not depict round metrics};
\end{tikzpicture}
\caption{A schematic of the completed manifold. The gray region is the compact cap, each contour represents a three-sphere, and the arrow follows a radial geodesic from $p$. The blue shell $R/2\leq r<R$ is a geodesic annulus by \eqref{eq:realballs}; its scalar integral gives the lower bound in \eqref{eq:divergence}. Both identities are proved in Section~\ref{sec:completion}.}
\label{fig:radial-shell}
\end{figure}

\par
The warping factor is close to the narrow cone $\rho=\alpha r$, but the factor $e^{K/\log r}$ makes it slightly concave. The small opening $\alpha$ magnifies the positive tangential Ricci endomorphism $\rho^{-2}\Ric_{\bar k_s}^{\sharp}$, with the index raised using $\bar k_s$, relative to errors of order $r^{-2}$. The slight concavity creates a positive radial term of order $r^{-2}(\log r)^{-2}$. In the induced slice metric $h_r=\rho^2\bar k_s$, the mixed covector has norm of order at most $r^{-2}(\log r)^{-1}$. Young's inequality bounds the mixed term by a portion of the positive tangential term and an error of order at most $r^{-2}(\log r)^{-2}$, which the radial margin absorbs.
The double logarithm makes the deformation derivatives small enough to fit these curvature scales, while $s(r)\to\infty$ preserves the divergence of the slice integral.

\subsection{All blocks of the Ricci tensor}
\relax{}
Put $x=\rho'/\rho$. For the outward unit normal $\partial_r$, we define the shape operator by raising one index of $\frac12\partial_rh_r$. Equivalently, $P(X)=\nabla^G_X\partial_r$. For this same normal, $P$ is the negative of the shape operator in Lee~\cite[pp.~235--236, Theorem 8.13(c)]{Lee2018}. Thus
\[
 P=x\Id+s'\mathcal A,\qquad H=\tr P=3x+s'\beta.
\]
Here and below, primes on $x$ and $s$ denote $r$ derivatives, whereas $\mathcal A_s=\partial_s\mathcal A$ and $\beta_s=\partial_s\beta$. The mean curvature $H$ is the trace, without division by the dimension. A superscript $\sharp$ on a Ricci tensor means that one index is raised with the indicated metric: $\mathcal T^\sharp$ and $\Ric_{h_r}^\sharp$ use $h_r^{-1}$, whereas $\Ric_{\bar k_s}^\sharp$ uses $\bar k_s^{-1}$. Finally, $\mathcal T=\Ric_G|_{T\Sph^3\times T\Sph^3}$ denotes the tangential block.

\begin{lemma}\label{lem:blocks}
The radial, mixed and tangential blocks of $\Ric_G$ are
\begin{align}
 \Ric_{rr}
   &=-3(x'+x^2)-(s''+2xs')\beta
                   -(s')^2(\beta_s+|\mathcal A|^2),\label{eq:rr}\\
 b:=\Ric_{r\,\cdot}
   &=s'\divg_{\bar k_s}\mathcal A,\qquad
 |b|_{h_r}^2=\rho^{-2}(s')^2|\divg\mathcal A|_{\bar k_s}^2,\label{eq:mixed}\\
 \mathcal T^\sharp
   &=\rho^{-2}\Ric_{\bar k_s}^\sharp
       -(x'+3x^2+\beta xs')\Id\notag\\
   &\hspace{8mm}-(s''+3xs'+\beta(s')^2)\mathcal A
       -(s')^2\mathcal A_s.\label{eq:tangent}
\end{align}
The sign in \eqref{eq:mixed} uses the stated outward normal and divergence convention. The coordinate calculation below fixes this sign; the positivity estimate uses the squared norm of $b$.
\end{lemma}
\begin{proof}
Take all parameter derivatives in fixed spatial charts.
\par
For a fixed slice metric $k$, the metric $\mathrm dr^2+\rho(r)^2k$ is a warped product in the sense of Bishop--O'Neill~\cite[Section 7]{BishopONeill1969}. Here $\bar k_{s(r)}$ also varies with $r$, producing additional deformation terms. We derive all contractions from the Gauss--Codazzi framework~\cite[Theorem 8.13]{Lee2018} with our convention. Use indices $i,j,k$ for the three tangential coordinates and Greek indices for all four coordinates; repeated indices are summed. Our Ricci convention, for which the unit sphere has positive curvature, is
\[
 \Ric_{\mu\nu}=\partial_\lambda\Gamma^\lambda_{\mu\nu}
 -\partial_\nu\Gamma^\lambda_{\mu\lambda}
 +\Gamma^\lambda_{\mu\nu}\Gamma^\sigma_{\lambda\sigma}
 -\Gamma^\sigma_{\mu\lambda}\Gamma^\lambda_{\nu\sigma}.
\]
Write $\mathrm{II}=\frac12\partial_rh_r=h_rP$. The nonzero symbols involving $r$ are $\Gamma^r_{ij}=-\mathrm{II}_{ij}$ and $\Gamma^i_{rj}=P^i{}_j$; all $\Gamma^\mu_{rr}$ and $\Gamma^r_{ri}$ vanish. The purely tangential symbols are those of $h_r$. The radial and mixed contractions are therefore
\begin{align*}
 \Ric_{rr}&=-\partial_r(P^i{}_i)-P^j{}_iP^i{}_j,\\
 \Ric_{ri}&=\partial_jP^j{}_i-\partial_iH
       +\Gamma^k_{jk}P^j{}_i-\Gamma^j_{ik}P^k{}_j
       =\nabla_jP^j{}_i-\partial_iH.
\end{align*}
In the tangential contraction, the terms whose indices are all tangential give $\Ric_{h_r}$. The remaining terms give
\[
 \mathcal T_{ij}=(\Ric_{h_r})_{ij}
       -\partial_r\mathrm{II}_{ij}-H\mathrm{II}_{ij}
       +2(h_rP^2)_{ij}.
\]
\[
 -\Gamma^r_{ik}\Gamma^k_{jr}=\mathrm{II}_{ik}P^k{}_j=(h_rP^2)_{ij},\qquad
 -\Gamma^k_{ir}\Gamma^r_{jk}=P^k{}_i\mathrm{II}_{jk}=(h_rP^2)_{ij}.
\]
In the last term of the Ricci formula, the two choices of radial index contribute the two copies of $h_rP^2$. Since $\partial_rh_r=2h_rP$, differentiating $\mathrm{II}=h_rP$ gives $\partial_r\mathrm{II}=2h_rP^2+h_rP'$. Thus these quadratic terms cancel before we raise an index. We obtain
\[
 \Ric_{rr}=-H'-\tr(P^2),\qquad
 \Ric_{ri}=(\divg P)_i-\partial_iH,\qquad
 \mathcal T^\sharp=\Ric_{h_r}^\sharp-P'-HP.
\]
\par
For \eqref{eq:tail}, the common volume gauge gives $\beta=\partial_s\log\bar v(s)$, so $H=3x+s'\beta$ is constant on each spatial slice; therefore the spatial derivative $\partial_iH$ in the Codazzi identity vanishes. On each fixed slice, $\rho(r)$ is a spatial constant. Thus $h_r=\rho^2\bar k_s$ has the same Levi--Civita connection as $\bar k_s$, $\Ric_{h_r}^\sharp=\rho^{-2}\Ric_{\bar k_s}^\sharp$, and covector squared norms scale by $\rho^{-2}$. Finally,
\[
 P'=x'\Id+s''\mathcal A+(s')^2\mathcal A_s,\qquad
 \tr(P^2)=3x^2+2xs'\beta+(s')^2|\mathcal A|^2.
\]
The derivative $P'$ is taken on the fixed tangent bundle. Moreover,
\[
 HP=(3x^2+xs'\beta)\Id+(3xs'+(s')^2\beta)\mathcal A.
\]
Also $H^\prime=3x^\prime+s^{\prime\prime}\beta+(s^\prime)^2\beta_s$. Since $\mathcal A$ is self-adjoint in the current metric, $\tr(\mathcal A^2)=|\mathcal A|^2$.
Substitution now gives \eqref{eq:rr} and \eqref{eq:tangent} term by term. For the mixed block, spatial constancy of $x,s',\beta$ gives $\divg P-\mathrm dH=s'\divg\mathcal A$, proving \eqref{eq:mixed} with the covector scaling already noted.
\end{proof}

The following radial identities follow by differentiating \eqref{eq:tail}. Primes on $x$ and $s$ denote $r$ derivatives. Since $\log\rho=\log\alpha+t+K/t$,
\[
 \frac{\dd t}{\dd r}=\frac1r,\qquad
 r^2x'=-1+\frac K{t^2}+\frac{2K}{t^3}.
\]
\begin{align}
 x&=\frac{1-K/t^2}{r},&
 x'+x^2&=-\frac{K(1-2/t-K/t^2)}{r^2t^2},\label{eq:radialderivatives}\\
 s'&=\frac1{rt},&
 s''&=-\frac1{r^2t}-\frac1{r^2t^2}.
 \label{eq:clockderivatives}
\end{align}
Since $t\geq T$, the constants in \eqref{eq:constants} also give
\begin{equation}\label{eq:warpingbounds}
 \alpha r\leq\rho(r)\leq\alpha e^{K/T}r,\qquad
 \rho^{-2}\geq\frac{5000}{r^2},\qquad
 0<x\leq\frac1r,\qquad \rho''<0.
\end{equation}
For the second inequality, use $\alpha^{-2}=10000$ and $e^{2K/T}=e^{1/5}\leq(1-1/5)^{-1}<2$, where the exponential bound follows by comparing its power series with the geometric series. The bounds for $x$ follow from $0<1-K/T^2\leq1-K/t^2\leq1$. For the last inequality, $\rho''/\rho=x'+x^2$ and $1-2/t-K/t^2\geq1-2/T-K/T^2>0$ in \eqref{eq:radialderivatives}.

\begin{proposition}\label{prop:tailricci}
The metric $G$ in \eqref{eq:tail} has positive definite Ricci tensor for every $r\geq L$.
\end{proposition}
\begin{proof}
We apply Lemma~\ref{lem:blocks} using the bounds \eqref{eq:slowbounds}, \eqref{eq:clockderivatives}, and \eqref{eq:warpingbounds}.

\par\smallskip\noindent\emph{The radial block.} Equation~\eqref{eq:clockderivatives} gives
\[
 s''+2xs'=\frac1{r^2t^2}\left(t-1-\frac{2K}{t}\right)>0,
\]
since $t\geq T=1000$ and $K=100$. Thus the term $-(s^{\prime\prime}+2xs^\prime)\beta$ in \eqref{eq:rr} is nonnegative. The remaining deformation term has absolute value at most $3/(r^2t^2)$ by \eqref{eq:slowbounds}. Consequently,
\begin{align}
 r^2t^2\Ric_{rr}
 &\geq3K(1-2/T-K/T^2)-3\notag\\
 &=\frac{29637}{100}>290.\label{eq:radialmargin}
\end{align}
\par\smallskip\noindent\emph{The tangential block.} $x\geq0$ implies $\beta xs'\leq0$, so its scalar contribution in \eqref{eq:tangent} is favorable. The remaining derivatives are
\begin{align*}
 r^2(x'+3x^2)&=2-\frac{5K}{t^2}+\frac{2K}{t^3}+\frac{3K^2}{t^4},\\
 r^2(s''+3xs')&=\frac2t-\frac1{t^2}
                                 -\frac{3K}{t^3}.
\end{align*}
The operator norm is at most the Hilbert--Schmidt norm in \eqref{eq:slowbounds}. Consequently the nonscalar part of the error satisfies
\[
 r^2\bigl\|(s''+3xs'+\beta(s')^2)\mathcal A
                  +(s')^2\mathcal A_s\bigr\|_{\rm op}
 \leq\frac2t+\frac5{t^2}+\frac{3K}{t^3}.
\]
Adding the scalar bound and dropping $-5K/t^2\leq0$ gives
\begin{equation}\label{eq:tangentmargin}
 \mathcal T^\sharp\geq\rho^{-2}\Ric_{\bar k_s}^\sharp
                          -\frac{D_*}{r^2}\Id,\qquad
 D_*=2+\frac2T+\frac5{T^2}+\frac{5K}{T^3}
                       +\frac{3K^2}{T^4}<3.
\end{equation}
Since $D_*<3<2500\kappa$, \eqref{eq:warpingbounds} and $\Ric_{\bar k_s}^{\sharp}\geq\kappa\Id$ give
\[
 \frac{D_*}{r^2}\Id
 \leq\frac\kappa2\rho^{-2}\Id
 \leq\frac12\rho^{-2}\Ric_{\bar k_s}^{\sharp}.
\]
Subtracting this from the intrinsic term in \eqref{eq:tangentmargin}, we obtain
\begin{equation}\label{eq:rhoestimate}
 \mathcal T^\sharp\geq\frac12\rho^{-2}\Ric_{\bar k_s}^{\sharp}
       \geq\frac{\kappa}{2}\rho^{-2}\Id>0.
\end{equation}
\relax{}
\par\smallskip\noindent\emph{The mixed block.} For $\xi=\tau\partial_r+v$, with $\tau\in\R$ and $v\in T\Sph^3$, equations \eqref{eq:mixed} and \eqref{eq:slowbounds} give $|b|_{h_r}^2\leq\rho^{-2}/(r^2t^2)$. By Cauchy--Schwarz and Young's inequality,
\[
 2|\tau b(v)|
 \leq\frac{\kappa}{4\rho^2}|v|_{h_r}^2
       +\frac{4\rho^2}{\kappa}|b|_{h_r}^2\tau^2
 \leq\frac{\kappa}{4\rho^2}|v|_{h_r}^2
       +\frac{4\tau^2}{\kappa r^2t^2}.
\]
Combining this with \eqref{eq:radialmargin} and \eqref{eq:rhoestimate} yields
\[
 \Ric_G(\xi,\xi)
 \geq\frac{\kappa}{4\rho^2}|v|_{h_r}^2
       +\frac{290-4/\kappa}{r^2t^2}\tau^2>0
 \qquad(\xi\ne0),
\]
since $\kappa=23/32$ and $290-4/\kappa>0$.
\end{proof}

\subsection{A scalar integral estimate}
Taking the trace of \eqref{eq:rhoestimate} gives the scalar integral estimate needed below.
\begin{proposition}\label{prop:scalartail}
For every $r\geq L$,
\begin{equation}\label{eq:scalarlower}
 \Scal_G\geq\frac12\rho^{-2}\Scal_{\bar k_{s(r)}},\qquad
 \int_{\Sph^3}\Scal_G\dd V_{h_r}
              \geq\frac12\rho(r)S(s(r)).
\end{equation}
\end{proposition}
\begin{proof}
\relax{}Taking the trace of \eqref{eq:rhoestimate} and using $\Ric_{rr}>0$ from \eqref{eq:radialmargin}, we obtain the pointwise estimate because $\Scal_G=\Ric_{rr}+\tr\mathcal T^\sharp$. Finally, $\rho$ is constant on each slice and $\dd V_{h_r}=\rho^3\dd V_{\bar k_s}$, so integration gives
\[
 \int_{\Sph^3}\Scal_G\dd V_{h_r}
 \geq\frac12\rho(r)\int_{\Sph^3}\Scal_{\bar k_s}\dd V_{\bar k_s}
 =\frac12\rho(r)S(s(r)).
\]
\end{proof}

\section{Smooth completion and divergence on metric balls}\label{sec:completion}

We attach a rotationally symmetric cap along the round collar at $r=L$, then estimate the scalar integral and volume of distance balls.

\begin{lemma}\label{lem:cap}
There exists a smooth function $f:[0,L]\to[0,\infty)$ such that $f(r)=r$ near zero, $f(r)=\rho(r)$ near $L$, and
\[
 f''\leq0,\qquad 0<f'\leq1.
\]
Consequently $\mathrm dr^2+f(r)^2\gamma$ extends smoothly across the center and has nonnegative sectional curvature.
\end{lemma}
\begin{proof}
Extend the formula for $\rho$ to a small left neighborhood of $L$. The strict inequality $\rho''(L)<0$ in \eqref{eq:warpingbounds} persists on a sufficiently small such neighborhood by continuity. Put
\[
 c=\frac{\rho(L)}L=\alpha e^{1/10},\qquad
 d=\rho'(L)=c(1-1/10000).
\]
Thus $0<d<c<1$. To construct $f$, it is enough to find a smooth nonnegative function $\zeta=-f''$, zero near zero and equal to $-\rho''$ near $L$, with moments
\begin{equation}\label{eq:moments}
 \int_0^L\zeta(r)\dd r=1-d=:M_c,\qquad
 \int_0^Lr\zeta(r)\dd r=\rho(L)-Ld=:P_c.
\end{equation}
The inequalities above imply $0<P_c<LM_c$, since $P_c/L=c-d<1-d$. Choose a smooth nonnegative collar function $\zeta_e$, equal to $-\rho''$ near $L$ and supported sufficiently close to $L$. After subtracting its moments from $(M_c,P_c)$, denote the residual mass and first moment by $M_{\mathrm{res}}$ and $P_{\mathrm{res}}$. We choose the collar so that $M_{\mathrm{res}}>0$ and its residual mean lies strictly between zero and the collar.

Indeed, if the collar is supported in $[L-\delta,L]$ and its two moments are $m_e,p_e$, then $m_e,p_e\to0$ as $\delta\to0$. Hence $(P_c-p_e)/(M_c-m_e)\to P_c/M_c<L$, whereas $L-\delta\to L$. For small enough $\delta$ the residual mean is positive and strictly less than $L-\delta$, leaving room for both interior bumps.
Such collar functions can be chosen explicitly: on $[L-\delta,L]$ use $\zeta_e(r)=-\rho''(r)\eta(2(r-L+\delta)/\delta)$ and extend by zero to the left. Flatness of $\eta$ makes the extension smooth, and it equals $-\rho''$ for $r\geq L-\delta/2$. On a fixed neighborhood of $L$, $|\rho''|$ is bounded; hence $m_e=O(\delta)$ and $0\leq p_e\leq Lm_e$, which justifies both limiting moments.

Choose two interior points $a_1<P_{\mathrm{res}}/M_{\mathrm{res}}<a_2$ below the collar, and smooth nonnegative symmetric bumps $\psi_1,\psi_2$ of unit integral, supported near $a_1,a_2$. Their first moments are $a_1,a_2$. Define the positive coefficients by
\[
 c_1=\frac{a_2M_{\mathrm{res}}-P_{\mathrm{res}}}{a_2-a_1},\qquad
 c_2=\frac{P_{\mathrm{res}}-a_1M_{\mathrm{res}}}{a_2-a_1}.
\]
Then $\zeta=\zeta_e+c_1\psi_1+c_2\psi_2$ has exactly the moments \eqref{eq:moments}. Define
\[
 f(r)=r-\int_0^r(r-\tau)\zeta(\tau)\dd\tau.
\]
The function $f$ equals $r$ near zero and satisfies $f'\in[d,1]$, so $f(r)\geq dr>0$ for $r>0$. Moreover, \eqref{eq:moments} gives $f(L)=\rho(L)$ and $f'(L)=\rho'(L)$. On the end collar, $f''=\rho''$; equality of value and derivative at $L$ therefore implies $f=\rho$ on the collar. The radial and tangential sectional curvatures of the cap are $-f''/f$ and $(1-(f')^2)/f^2$, respectively, both nonnegative. Near zero the metric is exactly Euclidean.
\par
These curvature formulas follow from the Gaussian connection in the proof of Lemma~\ref{lem:blocks}, now with $h_r=f^2\gamma$ and $P=(f'/f)\Id$. The radial contraction before taking a trace is $-(P'+P^2)=-(f''/f)\Id$. The Gauss equation subtracts $(f'/f)^2$ from the intrinsic sectional curvature $f^{-2}$, and the Codazzi terms vanish because $P$ is spatially constant and parallel. Thus, for an orthonormal frame $e_0=\partial_r,e_1,e_2,e_3$, with the last three vectors tangent to the slice,
\begin{align*}
 \langle R(e_0,e_i)e_j,e_0\rangle&=K_{\rm rad}\delta_{ij},\qquad
 \langle R(e_0,e_i)e_j,e_k\rangle=0,\\
 \langle R(e_i,e_j)e_k,e_l\rangle
 &=K_{\rm tan}(\delta_{jk}\delta_{il}-\delta_{ik}\delta_{jl}),\\
 K_{\rm rad}&=-f''/f,\qquad K_{\rm tan}=(1-(f')^2)/f^2.
\end{align*}

At $r>0$, the curvature operator acts by $K_{\rm rad}$ and $K_{\rm tan}$ on the two summands of $\Lambda^2T M=(\partial_r\wedge T\Sph^3)\oplus\Lambda^2T\Sph^3$. For orthonormal vectors $U=a\partial_r+X$ and $V=b\partial_r+Y$, with $X,Y$ tangent to the slice,
\[
 K(U,V)=K_{\rm rad}|aY-bX|^2+K_{\rm tan}|X\wedge Y|^2\geq0.
\]
\end{proof}

\begin{proof}[Proof of Theorem~\ref{thm:main}]
\emph{Completion and distance balls.} Use the cap of Lemma~\ref{lem:cap} for $r\leq L$ and the end \eqref{eq:tail} for $r\geq L$. Both metrics equal $\mathrm dr^2+\rho^2\gamma$ on a collar of $L$, so they define a smooth metric $G$ on $\R^4$, with center $p$. Lemma~\ref{lem:cap} and Proposition~\ref{prop:tailricci} give $\Ric_G\geq0$ everywhere, with strict positivity for $r\geq L$. The region $\{r\leq L\}$ is compact.
\par
Use the standard polar identification $(r,\omega)\mapsto r\omega$ with $\R^4\setminus\{0\}$, and set $p=0$. Each closed radial region $\{r\leq R\}$ is therefore a compact Euclidean ball. It carries the smooth positive definite metric just constructed, whose induced distance gives the same local topology: on any compact coordinate neighborhood, its smallest and largest eigenvalues relative to the Euclidean metric are bounded away from zero and infinity.
\par
For a piecewise smooth curve $\sigma$ on a compact parameter interval, $r\circ\sigma$ is absolutely continuous: away from $p$ the function $r$ is smooth, and near $p$ it is the Euclidean norm, hence locally Lipschitz. Along this curve, the inequality $|\dot r|\leq|\dot\sigma|_G$ holds almost everywhere; it follows from the Gaussian form away from $p$ and the Euclidean form near $p$; integration gives $|r(q)-r(q')|\leq d_G(q,q')$. A $d_G$-Cauchy sequence thus has bounded radii and lies in a compact radial region. It has a convergent subsequence in the common topology, and the Cauchy property then makes the whole sequence converge in $d_G$. This proves metric completeness. Since the metric is smooth on the connected finite-dimensional manifold $\R^4$, the Hopf--Rinow theorem~\cite[Theorem 6.19]{Lee2018} applies and yields geodesic completeness.

For a point with polar coordinates $(r,\omega)$, the radial curve from $p$ has length $r$, and the preceding lower bound shows that no shorter curve exists. Thus
\begin{equation}\label{eq:realballs}
 d_G(p,(r,\omega))=r,\qquad B_p(R)=\{r<R\}.
\end{equation}
The radial curves are geodesics because the metric has Gaussian form. On the Euclidean neighborhood of $p$, the variable $\omega\in\Sph^3$ is exactly the unit initial direction. The construction identifies $(0,\infty)\times\Sph^3$ diffeomorphically with $\R^4\setminus\{p\}$, and every slice metric is positive definite. Hence the geodesic with initial data $(p,\omega)$ is the radial curve for every nonnegative time; uniqueness of the geodesic equation gives $\exp_p(r\omega)=(r,\omega)$. This is a smooth bijection with the displayed smooth inverse, including at $r=0$ through the Euclidean cap. Thus $\exp_p:T_p\R^4\to\R^4$ is a diffeomorphism and $p$ is a pole.

\par\smallskip\noindent\emph{Divergence of the scalar integral.} For $R\geq2L$, define
\[
 m(R)=\inf_{r\in[R/2,R]}S(s(r)).
\]
Since $s(r)\to\infty$ and $S(s)\to\infty$, we have $m(R)\to\infty$: given $H>0$, all sufficiently large $r$ satisfy $S(s(r))\geq H$, and then the entire shell does so once $R/2$ is sufficiently large. No monotonicity of $S$ is needed. In Gaussian coordinates $\dd V_G=\dd r\,\dd V_{h_r}$. The boundary spheres of the shell have zero four-dimensional volume, so they do not affect the integral. Using $\Scal_G\geq0$ globally, the exact ball identity \eqref{eq:realballs}, $\rho\geq\alpha r$, and Proposition~\ref{prop:scalartail}, we obtain
\begin{align}
 R^{-2}\int_{B_p(R)}\Scal_G\dd V_G
 &\geq\frac{\alpha}{2R^2}\int_{R/2}^R rS(s(r))\dd r\notag\\
 &\geq\frac{3\alpha}{16}m(R)\longrightarrow+\infty.
 \label{eq:divergence}
\end{align}
Each integral is finite because the ball has compact closure and the metric is smooth.

\relax{}
\relax{}
\par\smallskip\noindent\emph{Volume growth and basepoints.} The slice-volume limit established in Subsection~\ref{sec:clock} and $s(r)\to\infty$ give $\Vol(\bar k_{s(r)})=2\pi^2\bar v(s(r))\to0$. Since $\rho(r)/r\to\alpha$, the slice volume $\mathcal V(r):=\Vol(\Sph^3,h_r)$ satisfies $\mathcal V(r)/r^3\to0$. Given $\epsilon>0$, choose $r_\epsilon\geq L$ with $\mathcal V(r)\leq\epsilon r^3$ for all $r\geq r_\epsilon$. Radial integration yields, for $R\geq r_\epsilon$,
\[
 \frac{\Vol_G B_p(R)}{R^4}
 \leq\frac{\Vol_G B_p(r_\epsilon)}{R^4}
       +\frac{\epsilon}{R^4}\int_{r_\epsilon}^R r^3\dd r
 \leq\frac{\Vol_G B_p(r_\epsilon)}{R^4}+\frac\epsilon4.
\]
Letting $R\to\infty$ and then $\epsilon\to0$ proves
\begin{equation}\label{eq:volumegrowth}
 \Vol_G B_p(R)=o(R^4),\qquad \operatorname{AVR}_p(G)=0.
\end{equation}
Finally fix any $q\in\R^4$ and put $d=d_G(p,q)$. For $R>d$, the triangle inequality gives $B_p(R-d)\subset B_q(R)\subset B_p(R+d)$. Nonnegative scalar curvature and \eqref{eq:divergence} imply
\begin{equation}\label{eq:basepoint}
 \mathcal Q_{G,q}(R)\geq
 \left(1-\frac dR\right)^2\mathcal Q_{G,p}(R-d)
 \longrightarrow+\infty.
\end{equation}
The other inclusion and \eqref{eq:volumegrowth} give $\Vol_G B_q(R)/R^4\leq(1+d/R)^4\Vol_G B_p(R+d)/(R+d)^4\to0$. Hence $\operatorname{AVR}_q(G)=0$ as well, completing the proof.
\end{proof}

\subsection{Euclidean products}

The extension to higher dimensions uses the elementary product argument also recorded by Hao--Zhu~\cite[Remark 1.2]{HaoZhu2026} for their three-dimensional example. Applied to Theorem~\ref{thm:main}, it retains a pole and divergence along all radii.

\begin{proposition}\label{prop:product}
For every integer $n\geq5$, there is a smooth complete metric $G_n$ on $\R^n$ with $\Ric_{G_n}\geq0$ and a pole, such that for every fixed $q\in\R^n$,
\[
 \lim_{R\to\infty}R^{2-n}\int_{B_q(R)}\Scal_{G_n}\dd V_{G_n}=+\infty.
\]
\end{proposition}
\begin{proof}
Set $m=n-4\geq1$ and take $G_n=G+g_{\R^m}$, where $G$ is the metric in Theorem~\ref{thm:main} and $g_{\R^m}$ is Euclidean. The product connection splits into the two factor connections. Hence $\Ric_{G_n}=\Ric_G\oplus0$, $\Scal_{G_n}=\Scal_G$ pulled back from the first factor, and the natural volume measure is the product measure.
The product distance satisfies
\[
 d_{G_n}((x,u),(y,v))^2=d_G(x,y)^2+|u-v|^2.
\]
The lower bound follows by integrating the norm of the pair of factor velocities; constant-speed minimizing geodesics in the factors give the reverse bound. Thus a product Cauchy sequence converges by completeness of both factors. The geodesic equations also split, giving $\exp_{(p,0)}(v,w)=(\exp_p(v),w)$, a global diffeomorphism. Therefore $(p,0)$ is a pole.
For any fixed $q=(x,u)$, the distance identity gives $B_G(x,R/2)\times B_{\R^m}(u,R/2)\subset B_{G_n}(q,R)$. Using nonnegative scalar curvature and integrating first in the Euclidean factor, we obtain
\begin{equation}\label{eq:product}
 R^{2-n}\int_{B_{G_n}(q,R)}\Scal_{G_n}\dd V_{G_n}
 \geq\frac{\omega_m}{2^{m+2}}\mathcal Q_{G,x}(R/2)
 \longrightarrow+\infty.
\end{equation}
The last limit is Theorem~\ref{thm:main}.
\end{proof}
The Euclidean directions have zero Ricci curvature, so these product examples do not retain the strict Ricci positivity outside a compact set of the four-dimensional construction.

\section*{Acknowledgements}
The author thanks Professor Bobo Hua for helpful discussions and support.

We would like to clarify that our work was completed on September 5, 2026, independently of the work of Hao and Zhu \cite{HaoZhu2026}. Their paper was submitted to arXiv on September 6, 2026, and was first made publicly available online on September 9, 2026. Therefore, our results were obtained independently before their work became publicly available. We also emphasize that the two constructions are different.
\section*{AI usage}
The author used OpenAI's Codex in this work. This assistance included proposing and refining candidate metric constructions and proofs, performing computational checks, and reviewing arguments through separate AI agents. The author guided the research direction, checked the mathematical statements and proofs, and carried out the final revisions. The author takes full responsibility for the content of the paper.\par
\bibliographystyle{amsplain}

\bibliography{references-en}

\end{document}